\documentclass[11pt]{amsart}

\usepackage{amsmath, amsthm}
\usepackage{amssymb}
\usepackage{amsfonts}
\usepackage{latexsym}
\usepackage{bm}

\theoremstyle{plain}
\newtheorem{theorem}{Theorem}[section]
\newtheorem{cor}[theorem]{Corollary}

\newtheorem{lem}[theorem]{Lemma}

\newtheorem{rem}[theorem]{Remark}

\newcommand{\hym}{hyperbolic metric} 
\newcommand{\WPm}{Weil-Petersson metric} 
\newcommand{\WP}{Weil-Petersson} 
\newcommand{\TS}{Teichm\"{u}ller space}

\newcommand{\Tm}{Teichm\"{u}ller metric}
\newcommand{\Td}{Teichm\"{u}ller distance}
\newcommand{\Tt}{Teichm\"{u}ller theory}

\newcommand{\kg}{Kleinian group}

\newcommand{\cmcs}{constant mean curvature surface}
\newcommand{\mc}{mean curvature}
\newcommand{\pc}{principal curvature}

\newcommand{\cbs}{conformal boundaries}
\newcommand{\cb}{conformal boundary}
\newcommand{\cs}{conformal structure}
\newcommand{\hqd}{holomorphic quadratic differential}
\newcommand{\RS}{Riemann surface}
\newcommand{\sff}{second fundamental form}
\newcommand{\af}{almost Fuchsian}
\newcommand{\qf}{quasi-Fuchsian}

\newcommand{\cc}{convex core}
\newcommand{\tm}{3-manifold}
\newcommand{\ms}{minimal surface}

\newcommand{\ps}{parallel surface}
\newcommand{\ef}{equidistant foliation}

\newcommand{\ppl}[2]{\frac{\partial{#1}}{\partial{#2}}}

\newcommand{\C}{\mathbb{C}}
\newcommand{\R}{\mathbb{R}}
\newcommand{\I}{\mathbb{I}}
\renewcommand{\H}{\mathbb{H}}

\newcommand{\Tcal}{\mathcal{T}}
\newcommand{\Fcal}{\mathcal{F}}

\numberwithin{equation}{section}
\allowdisplaybreaks

\makeatletter
\def\@citestyle{\m@th\upshape\mdseries}
\def\citeform#1{{\bfseries#1}}
\def\@cite#1#2{{%
  \@citestyle[\citeform{#1}\if@tempswa, #2\fi]}}
\@ifundefined{cite }{%
  \expandafter\let\csname cite \endcsname\cite
  \edef\cite{\@nx\protect\@xp\@nx\csname cite \endcsname}%
}{}
\makeatother

\begin{document}

\title{Quasi-Fuchsian 3-Manifolds and Metrics on Teichm\"{u}ller Space}
\author{Ren Guo}
\address{School of Mathematics, University of Minnesota, Minneappolis, MN 55455, USA}
\email{guoxx170@math.umn.edu}

\author{Zheng Huang}
\address{Department of Mathematics, The City University of New York, Staten Island, NY 10314, USA}
\email{zheng.huang@csi.cuny.edu}

\author{Biao Wang}
\address{Department of Mathematics, University of Toledo, Toledo, OH 43606, USA.}
\email{biao.wang@utoledo.edu}

\subjclass[2000]{Primary 30F60, 32G15; Secondary 53C42, 57M50}

\keywords{quasi-Fuchsian manifolds, foliation, minimal surfaces, Teichm\"uller distance, {\WPm}}

\begin{abstract}
An {\af} {\tm} is a {\qf} manifold which contains an incompressible closed {\ms} with {\pc}s in the range of $(-1,1)$. By the work of Uhlenbeck, 
such a {\tm} $M$ admits a foliation of {\ps}s, whose locus in {\TS} is represented as a path $\gamma$, we show that $\gamma$ 
joins the {\cs}s of the two components of the conformal boundary of $M$. Moreover, we obtain an upper bound for the {\Td} between 
any two points on $\gamma$, in particular, the {\Td} between the two components of the {\cb} of $M$, in terms of the {\pc}s of the {\ms} 
in $M$. We also establish a new potential for the {\WP} metric on {\TS}.
\end{abstract}
\maketitle
\section {Introduction of Main Results}

One of the fundamental questions in hyperbolic geometry of two and three dimensions is the interaction between the internal geometry 
of a hyperbolic {\tm} and the geometry of {\TS} of {\RS}s. It is natural to consider the situation for complete {\qf} {\tm}s. Let $M$ be such 
a manifold, then $M$ is topologically identified as $M = \Sigma \times \R$, where $\Sigma$ is a closed surface of genus $g \ge 2$. We 
denote the {\TS} of genus $g$ surfaces by $\Tcal_g(\Sigma)$, the space of {\cs}s on $\Sigma$ modulo orientation preserving 
diffeomorphisms in the homotopy class of the identity map. An important theorem of Brock (\cite{Br03}), proving a conjecture of Thurston, 
states that the hyperbolic volume of the {\cc} of a {\qf} {\tm} is quasi-isometric to the {\WP} distance between the two components of the 
{\cb} of the {\tm} in {\TS}.

The space of {\qf} {\tm}s $QF(\Sigma)$, called {\it the {\qf} space}, is a complex manifold of complex dimension $6g-6$. Its geometrical 
structures are extremely complicated, and they have been subjects of intensive study in recent years. In this paper, we consider mostly a 
subspace formed by the so-called {\it {\af} {\tm}s}: a {\qf} {\tm} $M$ is {\af} if it contains a unique embedded {\ms}, representing the 
fundamental group, whose {\pc}s are in the range of $(-1,1)$. This subspace is an open connected manifold of the 
same dimension (\cite {U83}). One can view the {\qf} space as a ``higher" {\TS}, a square with {\TS} sitting inside as a diagonal, and view the 
space of {\af} {\tm}s as an open subspace of $QF(\Sigma)$ around this diagonal. See also (\cite{KS07}) for generalization of almost Fuchsian 
manifolds to dS, AdS geometry and surfaces with cone points.

By a remarkable theorem of Uhlenbeck (\cite{U83}), any {\af} 3-manifold admits a foliation of parallel surfaces from the unique {\ms}. These 
{\ps}s, denoted by $S(r)$ for $r \in \R$, can be viewed as level sets of distant $r$ from the {\ms} (see \S 2.2 for details). Special hypersurfaces 
such as {\cmcs}s are investigated in \cite {Wang08} and \cite{HW09} and others. The existence of 
a foliation of submanifolds is an important and powerful tool in the study of geometry. A foliation of {\ps}s for an {\af} {\tm} allows one to 
relate the deformation of these structures to {\Tt}, the deformation theory of {\cs}s on a closed surface. We consider a foliation 
$\Fcal = \{S(r)\}_{r \in \R}$ of {\ps}s (or the {\it normal flow}) from the {\ms} $S$ for an {\af} {\tm} $M$ and we are interested in following 
the locus of this foliation in {\TS}: $\{S(r)\}_{r \in \R}$ of $M$ forms a path $\gamma(M)$ in ${\Tcal}_g(\Sigma)$. 
\begin{theorem}\label{path}
The path $\gamma(M)$ in ${\Tcal}_g(\Sigma)$ joins the {\cs}s of the two components of the {\cb} of $M$.
\end{theorem}   
Starting with any embedded surface with principal curvature in the range of $(-1,1)$ in a quasi-Fuchsian manifold, the foliation of {\ps}s 
from this surface forms a path in Teichm\"{u}ller space. It also joins the two components of conformal boundary. We establish this 
general situation in the proof, which is based on Esptein's construction in an unpublished manuscript (\cite{Eps84}).

Minsky (\cite {Min93}) considered the locus in {\TS} formed by pleated surfaces, and discovered that, in $\epsilon_0-$thick part 
of {\TS}, this locus is a quasi-geodesic in the {\Tm}.

In this paper, we answer a question raised by Rubstein, where he asked to relate the geometry of an {\af} {\tm} to distances 
in {\TS} (\cite{Rub07}). We relate this to the {\Tm} and the {\WPm}.

We obtain an upper bound of the {\Td} $d_T$ between the loci of the {\ms} $S$ and $S(r)$ for $r\in (-\infty,\infty)$. 
In particular, we obtain an upper bound of the {\Td} $d_T$ between the {\cs}s of fibers $S(r_1)$ and $S(r_2)$ for $r_1,r_2\in (-\infty,\infty)$.
\begin{theorem}\label{distance}
Let $\lambda_0 = max\{\lambda(z), z \in S\}$ be the maximal principal curvature on the {\ms} $S$, then,
\begin{eqnarray}
d_{T}(S,S(r)) &\le& \frac{1}{2}|\log\frac{1+\lambda_{0}\tanh(r)}{1-\lambda_0\tanh(r)}| \\
d_{T}(S(r_1),S(r_2)) &\le& \frac{1}{2}|\log\frac{1+\lambda_{0}\tanh(r_2)}{1-\lambda_0\tanh(r_2)}-
\log\frac{1+\lambda_{0}\tanh(r_1)}{1-\lambda_0\tanh(r_1)}|. 
\end{eqnarray}
\end{theorem}   
A direct consequence is the upper bound for the {\Td} between two {\cbs}.
\begin{cor}\label{boundary} 
If $M=\H^3/\Gamma$ is {\af} with the conformal boundary $\Omega_+/\Gamma \sqcup \Omega_-/\Gamma$, then
\begin{eqnarray}
d_{T}(S,\Omega_{\pm}/\Gamma) \le \frac{1}{2}\log\frac{1+\lambda_{0}}{1-\lambda_0}\\
d_{T}(\Omega_-/\Gamma,\Omega_+/\Gamma) \le \log\frac{1+\lambda_{0}}{1-\lambda_0}.
\end{eqnarray}
\end{cor}
One observes that if two {\cbs} of an {\af} {\tm} are far away in {\Tm}, then the maximal {\pc} of the unique {\ms} is very close to 1.

Epstein (\cite{Eps86}) calculated the maximal dilatation of the hyperbolic Gauss map of a surface homeomorphic to a disk in 
$\H^3$. The inequality $(1.4)$ was essentially due to him (\cite{Eps86} Proposition 5.1). Some local calculations in the proof of 
$(1.2)$ can also be seen in (\cite{Vel99}).

A rival metric on {\TS} to the {\Tm} is the {\WPm}, which offers a vastly different view of {\TS}: it is incomplete (\cite{Chu76}, 
\cite{Wol75}), K\"{a}hlerian (\cite{Ahl61}), and with negative sectional curvatures (\cite{Tro86}, \cite{Wol86}). We obtain a new 
potential for the {\WPm} on {\TS} by studying areas of certain immersed surfaces near Fuchsian locus within the {\qf} space.

A potential of {\WPm} is a function defined on {\TS} $\Tcal_g(\Sigma)$ such that the second variation of this function in direction of two 
{\hqd}s $\alpha dz^2$ and $\beta dz^2$ at a point of $\Tcal_g(\Sigma)$ equals the {\WP} inner product $\left\langle \alpha,\beta\right\rangle_{WP}$, 
up to a constant. Potentials offer important information and computational tools to study the metric. There are several known 
potentials of the {\WPm}: Dirichlet energy of harmonic maps (\cite{Tro87}, \cite{Wol89}, \cite{Jos91}), the Liouville action (\cite{ZT87}), 
renormalized volume of {\qf} manifolds (\cite{TT03}, \cite{KS08}) and Hausdorff dimension of the limit set (\cite{BT08}, \cite{McM08}). 
Those potentials are functions well-defined on the whole {\TS}. But to calculate the variation, it suffices that a function is defined locally 
in a neighborhood. 

We denote a Riemann surface by $(\Sigma,\sigma)$ where $\sigma$ is a conformal structure. 
Our new potential of the {\WPm} is a locally defined function in a neighborhood of $(\Sigma,\sigma)$ $\in \Tcal_g(\Sigma)$ which is based 
on the immersion of a {\ms} in some {\qf} manifold within the {\qf} space. The cotangent space of $\Tcal_g(\Sigma)$ at a point 
$(\Sigma,\sigma)$ can be identified with $Q(\sigma)$, the space of {\hqd}s on $(\Sigma,\sigma)$. Under this identification, 
$(\Sigma,\sigma)$ corresponds to the zero quadratic differential denoted by $0\in Q(\sigma)$. Uhlenbeck (\cite{U83}) showed that there is an open neighborhood $U$ of 
$0\in Q(\sigma)$ such that for any {\hqd} $\alpha dz^2\in U$, there exists a minimal surface $\Sigma(\alpha dz^2)$ immersed in some {\qf} {\tm} such 
that the induced metric on $\Sigma(\alpha dz^2)$ is in the conformal class $\sigma$ and the {\sff} of the immersion is the 
real part of $\alpha dz^2$. The area of $\Sigma(\alpha dz^2)$ with respect to the induced metric, denoted by $|\Sigma(\alpha dz^2)|$, is a function 
defined in the neighborhood $U \subset Q(\sigma)$. 
\begin{theorem}\label{wp} 
The zero quadratic differential $0\in Q(\sigma)$ is a critical point of the area functional $|\Sigma(\alpha dz^2)|:U\to {\R}$. And 
the area functional is a potential of {\WPm} on {\TS}. 
\end{theorem}
\vskip 0.1in
\noindent
{\bf Plan of the paper:} We provide necessary background material in \S 2. The Theorem \ref{path}, the Theorem \ref{distance} 
and the Corollary \ref{boundary} are proved in \S 3. In section \S 4, the Theorem \ref{wp} is established.
\vskip 0.1in
\noindent
{\bf Acknowledgements:} The authors wish to thank Bill Abikoff, Jun Hu, Albert Marden, and Xiaodong Wang for their 
generous help. We also thank the referee for suggestions to improve this paper. The research of the second named author is partially 
supported by a PSC-CUNY grant. 
\section{Preliminaries}
The results in this paper lie in the intersection of several different fields, low dimensional topology, geometric function theory and 
geometric analysis. In this section, we briefly summarize background material for topics involved. We always assume $M$ is orientable and all surfaces involved are closed and orientable of genus at least two.

\subsection{Kleinian groups and quasi-Fuchsian {\tm}s}
A hyperbolic manifold is a complete Riemannian manifold of constant curvature $-1$. The universal cover of a hyperbolic manifold is 
$\H^n$, and the deck transformations induce a representation of the fundamental group of the manifold in $Isom(\H^n)$, the 
(orientation preserving) isometry group of $\H^n$. In the case we are interested in, $Isom(\H^2) = PSL(2,\R)$ which lies naturally in 
$Isom(\H^3)= PSL(2,\C)$.

A subgroup $\Gamma \subset PSL(2,\C)$ is called a {\em Kleinian group} if $\Gamma$ acts on $\H^{3}$ properly discontinuously. For 
any {\kg} $\Gamma$, $\forall\,p\in\H^{3}$, the orbit set
\begin{eqnarray*}
   \Gamma(p)=\{\gamma(p)\ |\ \gamma\in \Gamma\}
\end{eqnarray*}
has accumulation points on the boundary $S^{2}_\infty=\partial\H^{3}$, and these points are the {\em limit points} of $\Gamma$, and 
the closed set of all these points is called the {\em limit set} of $\Gamma$, denoted by $\Lambda_{\Gamma}$. The complement 
of the limit set, i.e.,
\begin{eqnarray*}
   \Omega=S^{2}_{\infty}\setminus\Lambda_{\Gamma}\ ,
\end{eqnarray*}
is called the {\em region of discontinuity}. $\Gamma$ acts properly discontinuously on $\Omega$, whose quotient $\Omega/\Gamma$ 
is a finite union of {\RS}s of the finite type. The quotient $M =\H^{3}/\Gamma$ is a {\tm} with \it conformal boundary \rm
$\Omega/\Gamma$.

Suppose $\Gamma$ is a finitely generated torsion free {\kg} which has more than two limit points, we call $\Gamma$ {\em {\qf}} if its 
limit set $\Lambda_{\Gamma}$ is a closed Jordan curve and both components $\Omega_{\pm}$ of its region of discontinuity are 
invariant under $\Gamma$. When $\Lambda_{\Gamma}$ is actually a circle, we call $\Gamma$ a {\it Fuchsian group} and the 
corresponding $M$ Fuchsian manifold. It is clear that in this case, $M$ is the product space $\Omega_+/\Gamma \times \R$.

If $\Gamma$ is a {\qf} group, Marden \cite{Mar74} proved that the quotient $M=\H^{3}/\Gamma$ is diffeomorphic to
$(\Omega_+/\Gamma)\times(0,1)$, and $\overline{M}=(\H^{3} \cup\Omega)/\Gamma$ is diffeomorphic to 
$(\Omega_+/\Gamma)\times[0,1]$.  $M$ is then called a {\em quasi-Fuchsian {\tm}}.

In this paper we assume that $\Gamma$ contains no parabolic elements, and $M$ is {\qf} but we exclude the case when $M$ is 
Fuchsian since the locus of the normal flow from a totally geodesic {\ms} is a single point in {\TS}. Topologically 
$M = \Sigma \times \R$ where $\Sigma$ is a closed surface of genus $g \ge 2$. The deformation theory of {\kg} yields the following 
simultaneous unformization result, a beautiful theorem of Bers:
\begin{theorem} (\cite {Ber60b})
There is a biholomorphic map $Q: \Tcal_g (\Sigma)\times  \Tcal_g(\Sigma) \rightarrow QF(\Sigma)$ such that $Q(X,Y)$ is the 
unique {\qf} {\tm} in the {\qf} space $QF(\Sigma)$ with $X,Y$ as {\cbs}.
\end{theorem}

\subsection{Foliation of parallel surfaces}
In the following we review some results in \cite{U83}. Let $M$ be a {\qf} $3$-manifold which contains a {\ms} $S$ whose {\pc}s are in 
the range of $(-1,1)$. Then there is an {\ef} of $M$ formed by {\ps}s $\{S(r)\}_{r \in \R}$. Suppose the coordinate system on
$S\equiv S \times\{0\}$ is isothermal so that, using the local coordinate $z=x+iy,$ the induced metric $(dx,dy)[g_{ij}(z,0)]_{2\times 2}(dx,dy)^T$ on $S$ can be written in the form 
\begin{equation*}
   [g_{ij}(z,0)] = e^{2v(z)}\I
\end{equation*}
for some function $v(z)$ and where $\I$ is the $2\times{}2$ identity matrix. Let
$(dx,dy)[h_{ij}(z,0)]_{2\times 2}(dx,dy)^T$
be the {\sff} of $S$. We choose $\varepsilon>0$ to be sufficiently small, then the (local) 
diffeomorphism
\begin{align*}
   S \times(-\varepsilon,\varepsilon)&\to{}M\\
   (x^{1},x^{2},r)&\mapsto
   \exp_{x}(r\nu)
\end{align*}
induces a coordinate patch in $M$. Let $S(r)$ be the family of {\ps}s with respect to $S$, i.e.
\begin{equation*}
   S(r)=\{\exp_{x}(r\nu)\ |\ x\in{}S\}\ ,
   \quad{}r\in(-\varepsilon,\varepsilon)\ .
\end{equation*}

The induced metric $(dx,dy)[g_{ij}(z,r)](dx,dy)^T$ on $S(r)$ can be computed as
\begin{align}
[g_{ij}(z,r)]=&\ [g_{ij}(0,r)](\cosh{r}\I+\sinh{r}[g_{ij}(0,r)]^{-1}[h_{ij}(z,0)])^{2} \label{metric} \\
=&\ e^{2v(z)}(\cosh{r}\I+\sinh{r}e^{-2v(z)}[h_{ij}(z,0)])^{2}. \notag
\end{align}
And it is easy to verify, under the {\pc} condition, that induced metric on $S(r)$ (\ref{metric}) is well defined for all $r \in \R$. Moreover, the set of {\ps}s $\{S(r)\}_{r \in \R}$ forms a foliation of $M$, which is called the {\ef} or the normal flow.

The second fundamental form $(dx,dy)[h_{ij}(z,r)](dx,dy)^T$ on $S(r)$ is given by
\begin{equation*}
   h_{ij}(z,r)=\frac{1}{2}\,\ppl{}{r}\,g_{ij}(z,r)\ ,
   \quad{}1\leq{}i,j\leq{}2\ .
\end{equation*}
We denote $\lambda(z), -\lambda(z)$ as the {\pc}s of $S$ with $\lambda(z)\geq 0.$ Then the {\pc}s of $S(r)$ are given by
\begin{equation}\label{principal curvature}
   \lambda_{1}(z,r)=\frac{\tanh{}r-\lambda(z)}{1-\lambda(z)\tanh{}r}
   \qquad\text{and}\qquad
   \lambda_{2}(z,r)=\frac{\tanh{}r+\lambda(z)}{1+\lambda(z)\tanh{}r}\ .
\end{equation}
We observe that they are increasing functions of $r$ for fixed $z \in S$, and they both approach $\pm 1$ when 
$r \rightarrow \pm\infty$, respectively. 

\subsection{Metrics on {\TS}}

For references on quasiconformal mappings and Teichm\"{u}ller theory, see, for example, the paper \cite{Ber60a} and the book 
\cite{GL00}.

Points in {\TS} $\Tcal_g(\Sigma)$ are equivalence classes of {\cs}s (or alternatively, {\hym}s) on a closed surface $\Sigma$. Two Riemann surfaces $(\Sigma, \sigma_1)$ and $(\Sigma, \sigma_2)$ 
are equivalent if there is a homeomorphism $f: \Sigma \rightarrow \Sigma$, homotopic to the 
identity, which is a conformal map from $(\Sigma, \sigma_1)$ to $(\Sigma, \sigma_2)$. {\TS} carries a natural topology which is induced 
by several natural metrics, and among these metrics, the most studied are the {\Tm} and the {\WPm}.

The {\Td} between $(\Sigma, \sigma_1)$ and $(\Sigma, \sigma_2)$  is given by
\begin{equation}
  d_{T}((\Sigma, \sigma_1), (\Sigma, \sigma_2)) = {\frac{1}{2}}\log inf_{f}K[f],
\end{equation}
where $f:(\Sigma, \sigma_1)\to (\Sigma, \sigma_2)$ is a quasiconformal map homotopic to the identity and $K[f]$ is the 
maximal dilatation of $f$. This infimum is reached by the Teichm\"{u}ller map, which yields very geometrical description 
(\cite{Ber60a}).

The {\WPm} on {\TS} $\Tcal_g(\Sigma)$ is induced via an inner product on the cotangent space $Q(\sigma)$. At a point 
$(\Sigma, \sigma)\in \Tcal_g(\Sigma)$, let $\alpha dz^2,\beta dz^2 \in Q(\sigma)$ be two {\hqd}s on $(\Sigma, \sigma)$. 
The {\WP} co-metric is defined by the Hermitian form
\begin{equation}
\left\langle \alpha,\beta\right\rangle_{WP}= \int_\Sigma\frac{\alpha\overline{\beta}}{g_0^2}dA_{g_0},
\end{equation}
where $g_0|dz|^2$ is the unique hyperbolic metric in the conformal class $\sigma$, and $dA_{g_0}$ is its area element.

Infinitesimally, the Teichm\"{u}ller norm of a {\hqd} $\phi(z)dz^2\in Q(\sigma)$ is the $L^1$-norm of $\phi(z)dz^2$, while the 
{\WP} norm is induced by the $L^2$-norm. 

\section{Foliation as a Path in Teichm\"{u}ller Space}

In this section, we prove the theorems concerning the locus of the foliation $\{S(r)\}_{r\in \R}$ of {\ps}s for $M$ in {\TS}. In 
\S 3.1, we show the Theorem 1.1 and provide alternative point of view from geometric analysis. In \S 3.2, we calculate the 
dilatation of the quasi-conformal map from the {\ms} to a fiber $S(r)$, and in \S 3.3 we treat the general case, i.e, the 
quasi-conformal map between any two points on the locus.

\subsection{Joining the boundaries}

Let $S$ be the {\ms} in an {\af} manifold $M$, and $\{S(r)\}_{r\in \R}$ be the foliation of {\ps}s from $S$. The induced metric $(dx,dy)[g_{ij}(z,r)](dx,dy)^T$ on each fiber $S(r)$ is given by (\ref{metric}), and such metric determines a conformal structure $\sigma(r)$ on $S(r)$. The path 
$\gamma$ is formed by the set $\{(S(r), \sigma(r)) \in \Tcal_g(\Sigma)\}_{r\in \R}$. The Theorem 1.1 basically states that the {\cs}s 
determined by the limiting metrics are the {\cbs} $\Omega_+/\Gamma$ and $\Omega_-/\Gamma$, prescribed in Bers' 
Uniformization Theorem 2.1, where $M=\H^3/\Gamma$.

\begin{proof}[Proof of Theorem \ref{path}]
Consider the disjoint decomposition of the boundary of $\H^3$: $S_\infty^2=\Lambda \sqcup \Omega_+ \sqcup \Omega_-$, 
where $\Lambda_{\Gamma}$ is the limit set of $\Gamma$, and let $S \subset M$ be an embedded surface (not necessary to be minimal) with {\pc}s in 
the range of $(-1,1)$. The universal cover of $(M,S,\Omega_+/\Gamma,\Omega_-/\Gamma)$ is 
$(\H^3, \tilde{S}, \Omega_+, \Omega_-)$, where $\tilde{S}$ is topologically a disk with boundary $\Lambda_\Gamma$.

Note that we do not assume $S$ is a {\ms} in this proof, as long as it satisfies the {\pc} condition, or equivalently, the normal 
flow from $S$ can be extended to infinities. Then $M$ is isometric to $S \times \R$ with the metric 
\begin{equation}
dr^2 + (dx,dy)[g_{ij}(z,r)](dx,dy)^T,
\end{equation} 
where the second term is the induced metric on $S(r)$ as in (\ref{metric}), while $\{z\}$ is the conformal coordinate on $S$.

Let us make it clear how to follow the locus of $S(r)$ in {\TS}: for each $r$, the induced metric on $S(r)$ can be expressed in the 
complex form of $w(z,r)|dz+\mu(z,r)d\bar{z}|^2$. Solving the Beltrami equation $f_{\bar{z}} = \mu(z,r)f_z$ for each $r$ provides 
a quasi-conformal map from $S$ to $S(r)$ which is a diffeomorphism a.e., which associates a complex structure on $S(r)$, 
a point in {\TS} (\cite{Ber60a}).

We will work in the universal cover since in what follows is invariant under the action of $\Gamma$. We will use coordinates on 
$\Omega_+$ to write the embedding of $\tilde{S}$ into $\H^3$. $\H^3$ is identified with $B^3=\{(x,y,z)|x^2+y^2+z^2<1\}$ and 
$S_\infty^2$ with the unit sphere 
\begin{equation*}
S^2=\{X(\theta):=(\frac{2\Re\theta}{|\theta|^2+1},\frac{2\Im\theta}{|\theta|^2+1},\frac{|\theta|^2-1}{|\theta|^2+1})|\theta\in S_\infty^2\}.
\end{equation*}
Given a point $X(\theta)\in S^2$, there is unique horoshpere $H(\theta)$ at $X(\theta)$ tangent to  $\tilde{S}$. Therefore 
$\tilde{S}$ determines a function $\rho: S^2\to \mathbb{R}$ on $\Omega_+$ whose absolute value is the hyperbolic distance 
from $(0,0,0)\in B^3$ to $H(\theta)$, and $\rho(\theta)$ is positive when $(0,0,0)$ is out of $H(\theta)$, and it is negative when $(0,0,0)$ 
is in $H(\theta)$.

Using the function $\rho(\theta)$, Epstein (\cite{Eps84} (2.4)) obtained the embedding of $\tilde{S}$ in $B^3$:
\begin{eqnarray*}
\Omega_+ \to& B^3\\
\theta  \mapsto &\frac{|D\rho|^2+(e^{2\rho}-1)}{|D\rho|^2+(e^{\rho}+1)^2}X(\theta)+\frac{2D\rho}{|D\rho|^2+(e^{\rho}+1)^2},
\end{eqnarray*}
where $D\rho$ is the gradient of $\rho(\theta)$ with respect to the canonical metric on $S^2$: $\frac{4|d\theta|^2}{(1+|\theta|^2)^2}$.

With above formula of the embedding, one can work out explicitly the first and second fundamental forms of $\tilde{S}$ 
$$(d \Re\theta, d \Im\theta)[\overline{g}_{ij}(\theta,0)](d \Re\theta, d \Im\theta)^T,$$ 
$$(d \Re\theta, d \Im\theta)[\overline{h}_{ij}(\theta,0)](d \Re\theta, d \Im\theta)^T,$$
in the coordinate $\theta$ (see \cite{Eps84} (5.4, 5.6)).

As in the formula (\ref{metric}), the induced metric $(d \Re\theta, d \Im\theta)[\overline{g}_{ij}(\theta,r)](d \Re\theta, d \Im\theta)^T$ on the 
{\ps} $\widetilde{S(r)}$, with distance $r>0$ to $\tilde{S}$ is again 
\begin{equation*}
[\overline{g}_{ij}(\theta,r)]=[\overline{g}_{ij}(\theta,0)](\cosh{r}\I+\sinh{r}[\overline{g}_{ij}(\theta,0)]^{-1}[\overline{h}_{ij}(\theta,0)])^{2}.
\end{equation*}
When $r\to \infty$, the matrix $[\overline{g}_{ij}(\theta,r)]$ does not converge. Using the matrix $\cosh^{-2}{r}[\overline{g}_{ij}(\theta,r)]$ which induces 
the same {\cs} as $[\overline{g}_{ij}(\theta,r)]$, we get the limit (see \cite{Eps84} (5.10)):
\begin{align*}
\lim_{r\to \infty}\cosh^{-2}{(r)}[\overline{g}_{ij}(\theta,r)]=&\ [\overline{g}_{ij}(\theta,0)](\I+[\overline{g}_{ij}(\theta,0)]^{-1}[\overline{h}_{ij}(\theta,0)])^{2}\\ =&\ e^{2\rho(\theta)}\frac{4}{(1+|\theta|^2)^2}\I.
\end{align*}
We can check that, for any element $\gamma\in \Gamma$, the function $\rho(\theta)$ satisfies
$$e^{2\rho(\gamma(\theta))}\frac{|\gamma'(\theta)|^2}{(1+|\gamma(\theta)|^2)^2}=e^{2\rho(\theta)}\frac{1}{(1+|\theta|^2)^2}.$$
Therefore $e^{2\rho(\theta)}\frac{4|d\theta|^2}{(1+|\theta|^2)^2}$ defines a metric on $\Omega_+/\Gamma$ which is conformal under the 
coordinate $\theta$. Thus this metric
induces the same conformal structure on $\Omega_+/\Gamma$ coming from Bers's Uniformization.

It is clear that one can use above calculation for $r<0$ to treat the case of the {\RS} $\Omega_-/\Gamma$.
\end{proof}

We want to also provide an alternative approach to the Theorem 1.1, in a much more general setting. Let $\bar{N}$ be a 
compact manifold with boundary $\partial N$. We call $f$ a {\it defining function} if $f$ is a smooth function on $\bar{N}$ 
with a first order zero on $\partial N$. A Riemannian metric $g$ on $N = int(\bar{N})$ is called {\it conformally compact} if 
for any defining function $f$, $f^2g$ extends as a smooth metric on $\bar{N}$, whose restriction defines a metric and a 
well-defined {\cs} $C$ on $\partial N$. The pair $(\partial N, C)$ is the conformal infinity of $N$. A {\qf} manifold is a 
conformally compact Einstein manifold and the conformal infinity of a conformal compact Einstein manifold is 
independent of the choice of boundary defining functions (\cite{FG85}, \cite{GL91}). Therefore we need to find a 
defining function for the metric $g_M$ in formula $(3.1)$.

Let $t = \tanh(r)$, then $t \in (-1,1)$ and the formula $(3.1)$ can be written as
\begin{equation*}
g'_M = {\frac{1}{(1-t^2)^2}}dt^2 + {\frac{1}{1-t^2}}(dx, dy)h(z,t)(dx,dy)^T,
\end{equation*}
where $h(z,t) = e^{2v(z)}[\I+te^{-2v(z)}A(z)]^{2}$ and $z=x+iy$. Then $M$ is isometric to $\Sigma \times (-1,1)$ with the metric $g'_M$.

Now we define $f(t) = 1-t^2$, then $f(\pm 1) = 0$ and $f'(\pm 1) \not= 0$, and 
\begin{equation*}
f^{2}(t)g'_M = dt^2 +(1-t^2)(dx, dy)h(z,t)(dx,dy)^T.
\end{equation*}

It is then easy to see that $f^{2}(t)g'_M$ is a smooth metric on the compactified $\bar{M} = M \cup \partial M$, isometric 
to $\Sigma \times [-1,1]$, with conformal infinity the disjoint union of two {\cs}s achieved by taking $t = \pm 1$. Therefore 
by above definitions, $g(z,r)$, hence the {\ps}s, approach two components of the conformal infinity.
\begin{rem}
The question whether the locus $\gamma$ is simple in {\TS} is not known.
\end{rem}
\subsection{{\Td} bounds.} In this subsection, we prove formulas $(1.1)$ and $(1.3)$, i.e, we compare the {\cs} on $S(r)$ 
with the {\ms} $S$. Solutions are in a simpler form when we use conformal coordinates on a {\ms} since the {\mc} is zero.

The strategy is quite straightforward: we express the induced metric $(dx,dy)[g_{ij}(z,r)](dx,dy)^T$ on $S(r)$ in its complex form, then 
estimate the dilatation for the quasi-conformal map which assigns the complex structure on $S(r)$.

We fix the metric $(dx,dy)[g_{ij}(z,0)](dx,dy)^T=e^{2v(z)}dzd\overline{z}$ on $S$ and the {\sff} $(dx,dy)[h_{ij}(z,0)](dx,dy)^T$. Since the surface $S$ is minimal, we may write 
the matrix $e^{-2v(z)}[h_{ij}(z,0)]$ as 
\begin{equation*}
\begin{pmatrix}
                 a & b\\
                      b & -a
              \end{pmatrix}
\end{equation*}
with $\pm\lambda$ as eigenvalues, i.e., the principal curvatures of $S$. And we define $\lambda_0 > 0$ as the maximum of 
the {\pc}s of $S$. It is easy to see that $a^2+b^2 = \lambda^2$. 

Let $p^r$ be the intersection point of $S(r)$ ($r\in (-\infty, \infty)$) and the geodesic perpendicular to $S$ at $p$. There is a 
diffeomorphism $u^r:S\to S(r)$ sending $p$ to $p^r$.
\begin{lem}
This map $u^r:S\to S(r)$ is $\lambda_{0}|\tanh(r)|-$quasiconformal. 
\end{lem}
\begin{proof}
Assume $\rho^2d\zeta d\overline{\zeta}$ is the induced metric on $S(r)$ from the embedding $S(r)\to M$, in terms of the 
conformal coordinate $\zeta.$ To calculate the complex dilatation, we need to write the pull-back 
$(u^r)^*\rho^2d\zeta d\overline{\zeta}$ in its complex form $w(z,r)|dz+\mu(z,r)d\bar{z}|^2$ using the coordinate $z$ on $S$.

Write $z=x+iy$, $\alpha: = \cosh(r)$ and $\beta: = \sinh(r)$. We find that

\begin{eqnarray*}
(u^r)^*\rho^2d\zeta d\overline{\zeta}
        &=&(dx,dy)e^{2v(z)}(\cosh{r}\I+\sinh{r}e^{-2v(z)}[h_{ij}(z,0)])^{2}(dx,dy)^T \\
     &=&(dx,dy)e^{2v(z)}\begin{pmatrix}
                 \alpha+a\beta & b\beta\\
                      b\beta & \alpha-a\beta
              \end{pmatrix}^2(dx,dy)^T \\
                                   &=:&Edx^2+2Fdxdy+Gdy^2, 
\end{eqnarray*}
where 
\begin{eqnarray*}
E & = &e^{2v(z)} (\alpha^2+2a\alpha\beta+\beta^{2}\lambda^2)\\
 F &=&e^{2v(z)} (2b\alpha\beta) \\
 G&=& e^{2v(z)} (\alpha^2-2a\alpha\beta+\beta^{2}\lambda^2).
\end{eqnarray*}
From standard Riemannian geometry, we have 
\begin{equation}
\mu(z,r)  = \frac{E-G+2{\sqrt{-1}}F}{E+G+2{\sqrt{EG-F^2}}}.
\end{equation}

Here we have
\begin{eqnarray*}
EG-F^2 = e^{4v(z)}(\alpha^2 - \lambda^2 \beta^2)^2,
\end{eqnarray*}
and the complex dilatation of $u^r$ in a neighborhood of $p\in S$ is
\begin{eqnarray*}
\mu(z,r) &=& \tanh(r)(a + {\sqrt{-1}}b),
\end{eqnarray*}
with $|\mu(z,r)| = | \lambda(z)\tanh(r)|\leq \lambda_0|\tanh(r)|<1$. Since the surface $\Sigma$ is closed, 
$u^r$ is $\lambda_{0}|\tanh(r)|-$quasiconformal.
\end{proof}
Directly from $(2.3)$, the {\Td} between $S$ and $S(r)$ is bounded from above, i.e., 
\begin{eqnarray*}
d_{T} (S, S(r)) \le {\frac{1}{2}}\log{\frac{1+\lambda_0|\tanh(r)|}{1-\lambda_0|\tanh(r)|}} < 
{\frac{1}{2}}\log{\frac{1+\lambda_0}{1-\lambda_0}}.
\end{eqnarray*}
This completes the parts $(1.1)$ and $(1.3)$ in the Theorem 1.2, and the Corollary 1.3.
\subsection{{\Td} bounds: general case.} In this subsection, we complete the proof of the Theorem 1.2. The 
same strategy applies: given any two fiber surfaces $S(r_1)$ and $S(r_2)$, we determine the quasi-conformal 
map between them obtained from solving the Beltrami equation, and estimate the Beltrami coefficient.

There is a natural map from $u:S(r_1)\to S(r_2)$ given by the normal flow. More precisely, let $p \in S(r_1)$. Then 
$p' = u(p)$ is the intersection point of $S(r_2)$ and the geodesic perpendicular to $S(r_1)$ at $p$. 
\begin{lem}
The map $u$ is $k$-quasi-conformal with 
$k = \frac{\lambda_0|\tanh(r_2)-\tanh(r_1)|}{1-\lambda_0^2\tanh(r_2)\tanh(r_1)}$.
\end{lem}
\begin{proof}
Locally the metric on $S(r_1)$ is $e^{2v'}dz'd\overline{z'}$. The second fundamental form is $[h'_{ij}(z',0)]$ and we write 
the matrix $e^{-2v'(z')}[h'_{ij}(z',0)]$ as 
\begin{equation*}
\begin{pmatrix}
                 a & b\\
                 b & c
              \end{pmatrix}
\end{equation*}
with $\lambda_j$ as eigenvalues.

Therefore $a+c = \lambda_1+\lambda_2$ and $ac-b^2 = \lambda_1\lambda_2$.

The pull-back metric on $S(r_2)$ by $u$ is then given by
\begin{eqnarray*}
   &&(dx',dy')e^{2v'(z')}(\cosh(r_2-r_1)\I+\sinh(r_2-r_1)e^{-2v(z')}[h'_{ij}(z',0)])^{2}(dx',dy')^T \\
   &=& (dx',dy')e^{2v'(z')}\begin{pmatrix}
                 \alpha+a\beta & b\beta \\
                      b\beta & \alpha+c\beta
              \end{pmatrix}^2(dx',dy')^T\\
                        &=:&E'dx^2+2F'dxdy+G'dy^2.
             \end{eqnarray*}
            Here we write $\alpha = \cosh(r_2-r_1)$ and $\beta=\sinh(r_2-r_1)$, and 
\begin{eqnarray*}
E' & = &e^{2v'(z')} (\alpha^2+2a\alpha\beta+\beta^{2}(a^2+b^2))\\
 F' &=&e^{2v'(z')} (2b\alpha\beta + b\beta^2(a+c)) \\
 G'&=& e^{2v'(z')} (\alpha^2+2c\alpha\beta+\beta^{2}(c^2+b^2)).
\end{eqnarray*}

Now the metric $E'dx^2+2F'dxdy+G'dy^2$ can be written in the form $w'(z')|dz'+\mu(z')d\bar{z'}|^2$. 
It is easy to verify
\begin{equation}
E'G'-F'^2 = e^{4v'(z')}(\alpha + \lambda_1 \beta)^2(\alpha + \lambda_2\beta)^2,
\end{equation}
and 
\begin{equation}
E'+G'+2{\sqrt{E'G'-F'^2}} = e^{2v'(z')}(2\alpha + (\lambda_1+\lambda_2)\beta)^2,
\end{equation}
and
\begin{equation}
E'-G'+2\sqrt{-1}F'= e^{2v'(z')}\beta(2\alpha + (\lambda_1+\lambda_2)\beta)(a-c+2b\sqrt{-1}).
\end{equation}
Therefore, applying $(3.2)$, we obtain
\begin{eqnarray}
 \mu(z') & = & \frac{E'-G'+2{\sqrt{-1}}F'}{E'+G'+2{\sqrt{E'G'-F'^2}}} \nonumber \\
 &=& {\frac{\beta(a-c+2\sqrt{-1}b)}{2\alpha+(\lambda_1+\lambda_2)\beta}},
\end{eqnarray}
 and 
\begin{eqnarray}
 |\mu(z')| &=& {\frac{\beta|\lambda_2-\lambda_1|}{|2\alpha+(\lambda_1+\lambda_2)\beta|}} \nonumber \\
 &=& {\frac{|\lambda_2-\lambda_1|}{|\lambda_1+\lambda_2 + 2\coth(r_2-r_1)|}}.
\end{eqnarray}

On the other hand, as in the formula (\ref{principal curvature}), the {\pc}s 
$\lambda_1(z,r_1),\lambda_2(z,r_1)$ on $S(r_1)$ can be expressed in terms of the {\pc}s of the {\ms} 
$S$: $\pm\lambda(z)$. Thus, we find
\begin{equation}
|\mu(z')|=\frac{\lambda(z)|\tanh(r_2)-\tanh(r_1)|}{1-\lambda^2(z)\tanh(r_2)\tanh(r_1)}.
\end{equation}
Since $|\mu(z')|$ is an increasing function of $\lambda$ for fixed $r_1$ and $r_2$, we now find  
$$|\mu(z')|\leq\frac{\lambda_0|\tanh(r_2)-\tanh(r_1)|}{1-\lambda_0^2\tanh(r_2)\tanh(r_1)}$$ 
where $\lambda_0$ again is the maximum of {\pc}s on $S$.

It is now easy to see that
\begin{eqnarray*}
\frac{\lambda_0|\tanh(r_2)-\tanh(r_1)|}{1-\lambda_0^2\tanh(r_2)\tanh(r_1)}&<&\frac{|\tanh(r_2)-\tanh(r_1)|}{1-\tanh(r_2)\tanh(r_1)}\\
&=&|\tanh(r_2-r_1)|<1,
\end{eqnarray*}
and then $u$ is $\frac{\lambda_0|\tanh(r_2)-\tanh(r_1)|}{1-\lambda_0^2\tanh(r_2)\tanh(r_1)}-$quasiconformal.
\end{proof}
The Theorem 1.2 follows easily from the Lemma 3.3 and the Corollary 1.3 is obtained by triangle inequality or taking $r_1$ 
and $r_2$ to $-\infty$ and $\infty$.

\section{{\WP} Potential}

In this section, we prove the Theorem 1.4: the induced area functional of {\ms}s in the {\qf} space is a potential at the Fuchsian locus 
for the {\WP} metric on {\TS}. It is natural to consider the {\hym}s on a surface $\Sigma$ rather than the {\cs}s in the geometry of the 
{\WP} metric. A key fact for our theorem is that the {\sff} of a {\ms} is the real part of a {\hqd}.

Let  $g_0|dz|^2$ be the {\hym} on $(\Sigma,\sigma)$, where $\sigma$ is a {\cs} on $\Sigma$. By the Uniformization Theorem, the set of 
{\cs}s and the set of {\hym}s are of one-to-one correspondence. The hyperbolic surface $\Sigma$ with metric $g_0|dz|^2$ can be immersed into the 
Fuchsian manifold $\Sigma \times {\R}$ as a (totally geodesic) {\ms} with the vanishing {\sff}. 
\begin{lem}(\cite{U83})
There is an open neighborhood $U$ of $0 \in Q(\sigma)$ such that for any $\alpha dz^2 \in U$, there exists a {\ms} $\Sigma(\alpha dz^2)$, 
immersed in some {\qf} manifold such that the induced metric on $\Sigma(\alpha dz^2)$ is $e^{2u}g_0 |dz|^2$ and the {\sff} of the 
immersion is the real part of $\alpha dz^2$.
\end{lem}

\begin{proof}[Proof of Theorem \ref{wp}] We start with the classical Gauss equation. Using the fact that curvature of the {\hym} 
$g_0|dz|^2$ is $-1$, one can rewrite Gauss equation as the following quasi-linear elliptic equation:
\begin{equation}
\Delta_{g_0} u+1-e^{2u}-\frac{|\alpha|^2}{g_0^2}e^{-2u}=0,
\end{equation}
where $\Delta_{g_0}$ is the Laplace operator on the hyperbolic surface $\Sigma$ with metric $g_0|dz|^2$.

Now consider a one-parameter family of Gauss equations for metrics on a {\ms} in the 
conformal class of the {\hym} $g_0|dz|^2$, and {\sff} $\Re(t\alpha dz^2)$ for a fixed {\hqd} $\alpha dz^2 \in Q(\sigma)$.

For each small $t$, the solution $u^t$ to 
\begin{equation}\label{u1}
\Delta_{g_0} u^t+1-e^{2u^t}-\frac{|t\alpha|^2}{g_0^2}e^{-2u^t}=0.
\end{equation}
gives rise to an immersed {\ms} $\Sigma(t\alpha)$ with {\pc}s in the range of $(-1,1)$.

At $t=0$, the maximum principle implies the solution to
\begin{equation*}
\Delta_{g_0} u^0 = -1+ e^{2u^0}
\end{equation*}
is exactly $u^0 = 0$, hence $u^0=0$ corresponds to the totally geodesic case.

The (induced) area functional of $\Sigma(t\alpha dz^2)$ is given by
\begin{equation*}
|\Sigma(t\alpha dz^2)|=\int_\Sigma e^{2u^t}dA_{g_0}.
\end{equation*}

We denote $\dot{u} =  \frac{\partial u}{\partial t}|_{t=0}$. Differentiating (\ref{u1}) respect to $t$ and evaluate at $t=0$, we get 
\begin{equation*}
\Delta_{g_0} \dot{u}-2\dot{u}=0.
\end{equation*}

By the maximum principle, we have 
\begin{equation}\label{udot}
\dot{u} = 0.
\end{equation} 
The first variation of the area is then 
\begin{equation*}
\frac{\partial}{\partial t}|_{t=0}(|\Sigma(t\alpha dz^2)|) =\int_\Sigma 2e^{2u^0} \dot{u} dA_{g_0} = 0.
\end{equation*}
Hence the area functional is critical at the Fuchsian locus.

We now consider the second variation, let $t\alpha dz^2 +s\beta dz^2\in U\subset Q(\sigma)$ with $t,s\in {\C}$ and 
$|t|, |s|$ small enough. The area of $\Sigma(t\alpha dz^2+s\beta dz^2)$ is 
\begin{equation*}
|\Sigma(t\alpha dz^2+s\beta dz^2)|=\int_\Sigma e^{2u(t,s)} dA_{g_0}.
\end{equation*}
where $u(t,s)$ satisfies
\begin{equation}\label{u2}
\Delta_{g_0} u(t,s)+1-e^{2u(t,s)}-\frac{|t\alpha+s\beta|^2}{g_0^2}e^{-2u(t,s)}=0,
\end{equation}
with $u(0,0) = 0$.

Differentiating (\ref{u2}) respect to $t$ and $\overline{s}$ and evaluating at $t=s=0$, we obtain 
\begin{equation}\label{u3}
\Delta_{g_0} \frac{\partial^2 u}{\partial\overline{s} \partial t}|_{t=s=0}-
2\frac{\partial^2 u}{\partial\overline{s}\partial t}|_{t=s=0} - {\frac{\alpha\overline{\beta}}{g_0^2}}=0.
\end{equation}
Here we also used \eqref{udot}. Integrating the two sides of (\ref{u3}), we find
\begin{eqnarray*}
0&=&\int_\Sigma \Delta_{g_0} \frac{\partial^2 u}{\partial\overline{s} \partial t}|_{t=s=0} dA_{g_0} \\
&=&\int_\Sigma  2\frac{\partial^2 u}{\partial\overline{s}\partial t}|_{t=s=0}  dA_{g_0}
+ \int_\Sigma  {\frac{\alpha\overline{\beta}}{g_0^2}} dA_{g_0}. 
\end{eqnarray*}
Therefore 
\begin{eqnarray*}
&&\frac{\partial^2 }{\partial\overline{s} \partial t}|\Sigma(t\alpha dz^2+s\beta dz^2)| |_{t=s=0}\\
&=&\int_\Sigma   4 \frac{\partial u }{\partial\overline{s}} |_{t=s=0} \frac{\partial u }{\partial t} |_{t=s=0}dA_{g_0}
+  \int_\Sigma   2\frac{\partial^2 u}{\partial\overline{s}\partial t}|_{t=s=0}  dA_{g_0}\\
&=&- \int_\Sigma   {\frac{\alpha\overline{\beta}}{g_0^2}} dA_{g_0},
\end{eqnarray*}
here again we used \eqref{udot}. Now the second variation of the area functional is a constant multiple of the 
{\WP} pairing between {\hqd}s $\alpha dz^2$ and $\beta dz^2$ by $(2.4)$.
\end{proof}
\bibliographystyle{amsalpha}
\bibliography{ref}
\end{document}